\newcommand{\autosyncline}[1]{}
	\providecommand{\cvsId}{}

\providecommand{\cvsId}{\\\small\texttt{Id: funcfield0.tex,v 1.57 2009-01-27 22:18:33 jdemeyer Exp }}
\documentclass[12pt,a4paper,english]{article}
\newcommand{\Author}{Jeroen Demeyer%
\thanks{The author is a Postdoctoral Fellow of the Research Foundation -- Flanders (FWO).}}
\newcommand{\Title}{Hilbert's Tenth Problem for function fields over valued fields in characteristic zero}

\providecommand{\cvsId}{\\\small\texttt{Id: header.tex,v 1.55 2009-02-02 11:11:50 jdemeyer Exp }}
\usepackage[margin=25mm]{geometry}
\usepackage{babel}
\usepackage{amsmath,amsfonts,amsthm,amssymb}
\usepackage{color}
\usepackage{ifthen}
\usepackage{enumerate}
\usepackage{picins}
\usepackage[mathlines,displaymath]{lineno}  

	\usepackage[pdftex]{graphics}
	\usepackage[pdftex,pdfstartview={FitH},bookmarks]{hyperref}
\newcommand{\twodigit}[1]{\ifthenelse{#1 < 10}{0#1}{#1}}
\providecommand{\Date}{\number\year--\twodigit{\number\month}--\twodigit{\number\day}}
\providecommand{\cvsId}{}
\author{\Author}
\title{\Title}
\date{\Date\cvsId}

\ifthenelse{\equal{\languagename}{dutch}}{
	\providecommand{\theoremname}{Stelling}
	\providecommand{\maintheoremname}{Hoofdstelling}
	\providecommand{\conjname}{Conjectuur}
	\providecommand{\openname}{Open Probleem}
	\providecommand{\questionname}{Vraag}
	\providecommand{\propname}{Propositie}
	\providecommand{\obsname}{Vaststelling}
	\providecommand{\lemmaname}{Lemma}
	\providecommand{\mainlemmaname}{Hoofdlemma}
	\providecommand{\corname}{Gevolg}
	\providecommand{\claimname}{Bewering}
	\providecommand{\definename}{Definitie}
	\providecommand{\examplename}{Voorbeeld}
	\providecommand{\cexamplename}{Tegenvoorbeeld}
	\providecommand{\exercisename}{Oefening}
	\providecommand{\factname}{Feit}
	\providecommand{\factsname}{Feiten}
	\providecommand{\remarkname}{Opmerking}
	\providecommand{\solutionname}{Oplossing}
	\providecommand{\stepname}{Stap}
	
}{}
\providecommand{\theoremname}{Theorem}
\providecommand{\maintheoremname}{Main Theorem}
\providecommand{\conjname}{Conjecture}
\providecommand{\openname}{Open Problem}
\providecommand{\questionname}{Question}
\providecommand{\propname}{Proposition}
\providecommand{\obsname}{Observation}
\providecommand{\lemmaname}{Lemma}
\providecommand{\mainlemmaname}{Main Lemma}
\providecommand{\corname}{Corollary}
\providecommand{\claimname}{Claim}
\providecommand{\definename}{Definition}
\providecommand{\examplename}{Example}
\providecommand{\cexamplename}{Counterexample}
\providecommand{\exercisename}{Exercise}
\providecommand{\factname}{Fact}
\providecommand{\factsname}{Facts}
\providecommand{\remarkname}{Remark}
\providecommand{\solutionname}{Solution}
\providecommand{\stepname}{Step}

\theoremstyle{plain}
\ifx\chaptername\undefined
	\newtheorem{theorem}{\theoremname}
\else
	\newtheorem{theorem}{\theoremname}[chapter]
\fi
\newtheorem*{theorem*}{\theoremname}

\newtheorem{mtheorem}[theorem]{\maintheoremname}
\newtheorem*{mtheorem*}{\maintheoremname}

\newtheorem*{conj*}{\conjname}

\newtheorem*{open*}{\openname}
\newtheorem{prop}[theorem]{\propname}
\newtheorem*{prop*}{\propname}

\newtheorem*{obs*}{\obsname}
\newtheorem{lemma}[theorem]{\lemmaname}
\newtheorem*{lemma*}{\lemmaname}

\newtheorem*{mlemma*}{\mainlemmaname}
\newtheorem{cor}[theorem]{\corname}
\newtheorem*{cor*}{\corname}

\theoremstyle{definition}
\newtheorem{define}[theorem]{\definename}
\newtheorem*{define*}{\definename}
\newtheorem{example}[theorem]{\examplename}
\newtheorem*{example*}{\examplename}

\newtheorem*{cexample*}{\cexamplename}

\newtheorem*{question*}{\questionname}

\newtheorem*{exercise*}{\exercisename}

\newtheorem*{fact*}{\factname}

\newtheorem*{facts*}{\factsname}

\theoremstyle{remark}
\newtheorem*{remark}{\remarkname}

\newcounter{step}

\newcounter{case}

\newcommand{\inv}{^{-1}}
\newcommand{\iinv}{^{-2}}
\newcommand{\ellzero}{\mathbf{0}}
\newcommand{\hens}{^\mathrm{H}}

\newcommand{\Gb}{\Gamma\!}
\newcommand{\eps}{\varepsilon}

\newcommand{\frakm}{\mathfrak{m}}

\newcommand{\C}{\mathbb{C}}

\newcommand{\N}{\mathbb{N}}
\newcommand{\Q}{\mathbb{Q}}
\newcommand{\R}{\mathbb{R}}
\newcommand{\Z}{\mathbb{Z}}

\newcommand{\calB}{\mathcal{B}}

\newcommand{\calE}{\mathcal{E}}

\newcommand{\calH}{\mathcal{H}}

\newcommand{\calL}{\mathcal{L}}

\newcommand{\calO}{\mathcal{O}}

\newcommand{\calR}{\mathcal{R}}
\newcommand{\calS}{\mathcal{S}}

\newcommand{\chebX}{\mathrm{X}}
\newcommand{\chebY}{\mathrm{Y}}

\newcommand{\AS}[2][void]{\ifthenelse{\equal{#1}{void}}{\mathbb{A}\!^{#2}}{\mathbb{A}\!^{#2}(#1)}}
\newcommand{\PS}[2][void]{\ifthenelse{\equal{#1}{void}}{\mathbb{P}^{#2}}{\mathbb{P}^{#2}(#1)}}

\newcommand{\FF}[1]{\mathbb{F}\!_{#1}}

\DeclareMathOperator{\fchar}{char}

\DeclareMathOperator{\End}{End}

\DeclareMathOperator{\gal}{Gal}

\DeclareMathOperator{\cdim}{cd}

\newcommand{\set}[2]{\{{#1}\mid{#2}\}}

\newcommand{\form}[1]{\langle{#1}\rangle}

\newcommand{\into}{\hookrightarrow}
\newcommand{\onto}{\twoheadrightarrow}
\newcommand{\simto}{\stackrel{\sim}{\to}}

\ifx\needsha\undefined
\else
	\DeclareFontEncoding{OT2}{}{}

\fi

\usepackage[all]{xy}

\begin{document}

\maketitle

\begin{abstract}
\autosyncline{20\\funcfield0}Let $K$ be a field with a valuation satisfying the following conditions:
both $K$ and the residue field $k$ have characteristic zero;
the value group is not $2$-divisible;
there exists a maximal subfield $F$ in the valuation ring
such that $\gal(\bar{F}/F)$ and $\gal(\bar{k}/k)$
\autosyncline{25}have the same $2$-cohomological dimension and this dimension is finite.
Then Hilbert's Tenth Problem has a negative answer for any function field of a variety over $K$.
In particular, this result proves undecidability for varieties over $\C((T))$.
\end{abstract}

\section{Introduction}

\autosyncline{34\\funcfield0}Hilbert's Tenth Problem (from his famous list of $23$ problems)
is the following:
find an algorithm which, given a polynomial $f \in \Z[X_1, \dots, X_n]$,
decides whether or not $f$ has a zero in $\Z^n$.
It has been shown that such an algorithm does not exist
\autosyncline{39}by Matiyasevich (see \cite{matiy-h10}), building on earlier work by Davis, Putnam and Robinson.
See \cite{davis-h10} for a survey article with the proof of Hilbert's Tenth Problem.

\autosyncline{42}Hilbert's Tenth Problem (HTP) can be generalized as follows:
let $\calR$ be a ring and $\calR_0$ a finitely generated $\Z$-algebra in $\calR$.
Then Hilbert's Tenth Problem for $\calR$ with coefficients in $\calR_0$
is the question whether there exists an algorithm
which can decide whether a polynomial $f \in \calR_0[X_1, \dots, X_n]$
\autosyncline{47}has a solution in $\calR^n$.

\autosyncline{49}The ring $\calR_0$ is called the \emph{coefficient ring}.
If $\calR$ is a field,
then the equation $f(X_1, \dots, X_n) = 0$ is equivalent to $c f(X_1, \dots, X_n) = 0$
for $c \in \calR \setminus \{0\}$.
Therefore, we might as well take coefficients in the fraction field of $\calR_0$.
\autosyncline{54}In this paper, $\calR$ will always be a field
and we will take $\calR_0$ to be a finitely generated subfield of $\calR$.

\begin{remark}
\autosyncline{58}Technically speaking we do not really need $\calR_0$ to be finitely generated,
but finitely generated rings have many nice properties.
For more general rings, we would have to be more careful
with our definition of Hilbert's Tenth Problem and diophantine models.
Furthermore, all undecidability results so far (except for $\calR = \Z$)
\autosyncline{63}work by interpreting $\Z$ in $\calR$;
this interpretation involves finitely many polynomials
and hence it suffices to adjoin these finitely many coefficients to the ring $\calR_0$.
\end{remark}

\autosyncline{68}This paper deals with Hilbert's Tenth Problem
for function fields over valued fields,
where both the valued field and the residue field have characteristic zero,
the value group is not $2$-divisible and some condition on Galois cohomology is satisfied
(see Main Theorem~\ref{main-valued2}).

\autosyncline{74}Our Main Theorem
generalizes a result by Kim and Roush (see \cite{kim-roush-ct1t2}),
who proved the negative answer to HTP for $\C(Z_1, Z_2)$ (with coefficients in $\Q(Z_1, Z_2)$).
Eisentr\"ager extended this to function fields of varietes of dimension $\geq 2$ over $\C$
(see \cite{eisentraeger-funcfield0}).

\autosyncline{80}There are already a lot of results on HTP for function fields:
Denef proved undecidability for rational function fields over real fields
(see \cite{denef-real}),
Moret-Bailly generalized this to function fields of varieties over real fields
(see \cite{mb-ell}).
\autosyncline{85}Kim and Roush proved the negative answer to HTP
for rational function fields over $p$-adic fields
(subfields of $\Q_p$, including all number fields).
This was generalized to function fields of varieties
independently by Moret-Bailly (see \cite{mb-ell}) and Eisentr\"ager (see \cite{eisentraeger-padic}).
\autosyncline{90}In positive characteristic,
Pheidas proved undecidability for $\FF{q}(Z)$ (see \cite{pheidas5}) with $q$ odd,
Videla did the same for $q$ even (see \cite{videla-ff2}).
This was generalized to function fields over finite fields
by Shlapentokh (see \cite{shlapentokh-ffff}) and Eisentr\"ager (see \cite{eisentraeger-ffff2}).
\autosyncline{95}One of the biggest open questions regarding function fields is $\C(Z)$.

\autosyncline{97}For our result, we consider function fields of curves over valued fields with residue characteristic zero.
So we cannot apply our result to $\Q_p(Z)$ for example.
One important application of our result where HTP was not known before
is the field $\C((T))(Z)$.

\textit{Acknowledgements.}
\autosyncline{103}A large part of this article was written during a stay at
the Scuola Normale Superiore di Pisa,
funded by the Research Foundation -- Flanders (FWO).
I thank the University of Ghent, the FWO and the Scuola Normale Superiore
for giving me this possibility.
\autosyncline{108}I thank Karim Becher for telling me about the Milnor conjectures.
I also thank Laurent Moret-Bailly and Angelo Vistoli
for some helpful discussions regarding algebraic geometry.

\section{Preliminaries}

\autosyncline{116\\funcfield0}Before we can state the Main Theorem (see Section~\ref{sec-valmain1} and Section~\ref{sec-valmain2}),
we need some definitions regarding diophantine sets, valuations and quadratic forms.

\subsection{Diophantine sets and diophantine models}\label{sec-dioph}

\autosyncline{123\\funcfield0}The most important definition in the study of Hilbert's Tenth Problem
is that of a diophantine set:
\begin{define}
Let $\calR_0 \subseteq \calR$ be rings.
Let $\calS$ be a subset of $\calR^n$.
\autosyncline{128}Then $\calS$ is called \emph{diophantine} over $\calR$ with coefficients in $\calR_0$
if and only if there exists a polynomial $f \in \calR_0[A_1, \dots, A_n, X_1, \dots, X_m]$
for some $m \geq 0$ such that
$$
	\calS = \set{(a_1, \dots, a_n) \in \calR^n}
		{f(a_1, \dots, a_n, x_1, \dots, x_m) = 0 \text{ for some } (x_1, \dots, x_m) \in \calR^m}
.$$
\end{define}

\autosyncline{137}Next, we need to define a diophantine model of one ring $\calS$ over a ring $\calR$.
This is a way of encoding the ring $\calS$ as elements of $\calR$
in a diophantine way.
\begin{define}
Let $\calS$ and $\calR_0 \subseteq \calR$ be rings.
\autosyncline{142}A \emph{diophantine model} of $\calS$ over $\calR$ with coefficients in $\calR_0$
is an injective map $\phi: \calS \into \calR^m$ for some $m \geq 1$
such that the following sets are diophantine with coefficients in $\calR_0$:
\begin{enumerate}
\item The image $\phi(\calS) \subseteq \calR^m$.
\item \autosyncline{147}The graph of addition $\set{(\phi(x), \phi(y), \phi(x+y))}{x,y \in \calS} \subseteq \calR^{3m}$.
\item The graph of multiplication $\set{(\phi(x), \phi(y), \phi(x y))}{x,y \in \calS} \subseteq \calR^{3m}$.
\end{enumerate}
\end{define}

\autosyncline{152}The reason for this definition is the following reduction,
which is usually applied with $\calS_0 = \calS = \Z$:

\begin{prop}\label{diophmodel}
\autosyncline{156}Let $\calS_0 \subseteq \calS$ and $\calR_0 \subseteq \calR$ be rings
such that $\calS_0$ and $\calR_0$ are finitely generated $\Z$-algebras.
Assume $\phi: \calS \into \calR^m$ is a diophantine model
such that $\phi\inv(\calR_0^m)$ contains a set of generators of $\calS_0$.
If HTP for $\calS$ with coefficients in $\calS_0$ has a negative answer,
\autosyncline{161}then also HTP for $\calR$ with coefficients in $\calR_0$ has a negative answer.
\end{prop}

\subsection{Valuations}\label{sec-val}

\autosyncline{168\\funcfield0}In this section
we give definitions and properties of (Krull) valuations.
We refer to \cite{endler} and \cite{engler-prestel}.

\begin{define}
\autosyncline{173}A \emph{totally ordered $\Z$-module} $\Gamma$
is a $\Z$-module with a total order $\leq$
such that $a \leq b$ implies $a+c \leq b+c$ for all $a,b,c \in \Gamma$.
\end{define}

\autosyncline{178}Remark that totally ordered $\Z$-modules are always torsion-free.

\begin{define}
\autosyncline{181}A \emph{valuation} $v$ on a field $K$ is a surjective map $v: K^* \onto \Gamma$,
where $\Gamma$ is a totally ordered $\Z$-module,
satisfying the following conditions:
\begin{enumerate}
\item For all $x, y \in K^*$, $v(xy) = v(x) + v(y)$.
\item \autosyncline{186}For all $x, y \in K^*$ such that $x + y \neq 0$, $v(x + y) \geq \min(v(x), v(y))$.
\end{enumerate}
$\Gamma$ is called the \emph{value group} of the valuation.
Usually one defines $v(0) = \infty$, which is consistent with the above axioms
if $\infty$ is treated as an element greater than any element from $\Gamma$.
\end{define}

\autosyncline{193}Every field has a \emph{trivial valuation} with value group $\{0\}$.
Then $v(x) = 0$ for $x \in K^*$ and $v(0) = \infty$.

\autosyncline{196}If $v: K^* \onto \Gamma$ is a valuation, the \emph{valuation ring} $\calO$
is the ring consisting of all elements of $K$ having non-negative valuation:
$$
	\calO = \set{x \in K}{v(x) \geq 0}
.$$

\autosyncline{202}In $\calO$, the elements with strictly positive valuation form a \emph{maximal ideal} $\frakm$.
The field $k := \calO/\frakm$ is called the \emph{residue field} of $K$ with respect to $v$.
We have a natural surjection $\pi: \calO \onto k$.
Note that for all $x \in K$, either $x \in \calO$ or $x\inv \in \calO$.
The elements for which both hold
\autosyncline{207}form the \emph{unit group} $\calO^*$,
the set of elements with valuation equal to zero.
We have a short exact sequence $1 \to \calO^* \to K^* \stackrel{v}{\to} \Gamma \to 0$.
This shows that
the ring $\calO$ determines completely the value group and the valuation.

\begin{prop}\label{valuation-ext}
\autosyncline{214}Let $K$ be a field with a valuation $v$
such that $K$ and its residue field have characteristic zero.
Let $L$ be a finite extension of $K$
and let $v_1, \dots, v_n$ denote all the extensions of $v$ to $L$.
Let $e_i$ denote the respective ramification indices and $f_i$ the residue extension degrees.
\autosyncline{219}Then $\sum_{i=1}^n e_i f_i = [L:K]$.
\end{prop}

\begin{proof}
\autosyncline{223}This follows from Corollary~(20.23)
and the definition of defectless at the beginning of \S18 in \cite{endler}.
\end{proof}

\begin{remark}
\autosyncline{228}In general, the equality in Proposition~\ref{valuation-ext}
is only an inequality $\sum_{i=1}^n e_i f_i \leq [L:K]$ because of
possible inseperability either in $K$ or in the residue field.
However, if the value group is $\Z$, then the equality holds anyway
if $L/K$ is seperable.
\end{remark}

\begin{define}
\autosyncline{236}With notations as above,
a valued field $K$ is called \emph{henselian}
if and only if the following property (called \emph{Hensel's Lemma}) holds:

\autosyncline{240}For every $P \in \calO[Z]$ and $\alpha \in k$
such that $\alpha$ is a simple root of $P \bmod \frakm$,
there exists a $\beta \in \pi\inv(\alpha) \subseteq \calO$ such that $P(\beta) = 0$
(the simple root $\alpha$ in the reduction can be lifted to a global root $\beta$).
\end{define}

\autosyncline{246}If $K$ is a field with valuation $v$, the \emph{henselisation} $K\hens$
is the smallest extension of $K$ which is henselian.
This always exists and is an algebraic extension of $K$
(it is usually defined as the fixed field of a certain subgroup
of $\gal(K^\text{sep}/K)$).
\autosyncline{251}Given an algebraic closure $\bar{K}$,
the henselisation $K\hens$ is a uniquely defined subfield of $\bar{K}$.
The henselisation is an immediate extension,
i.e.\ the value group $\Gamma$ and the residue field $k$ remain the same.
All this follows from \cite[Section~5.2]{engler-prestel}.

\begin{prop}\label{reshensel}
\autosyncline{258}Let $K$ be a valued field with notations as above.
If $K$ is henselian and $\fchar K = \fchar k = 0$,
then $\calO$ contains a maximal subfield $F$.
The projection $\pi$ maps $F$ isomorphically onto $k$.
\end{prop}

\begin{proof}
\autosyncline{265}We give a sketch of the proof,
see \cite[Lemma~5.4.13~(ii)]{chang-keisler} for more details.
\parpic[r]{
$$\xymatrix@R=1.5em{
	K \\
	\calO \ar@{-}[u] \ar@{->>}[r]^\pi & k \\
	F \ar@{-}[u] \ar[r]^\sim & F^\pi \ar@{-}[u] \\
	\Q \ar@{-}[u] & \Q \ar@{-}[u]
}$$
}

\autosyncline{276}Since $\fchar k = 0$, the valuation will be trivial on $\Q$,
so $\calO$ contains $\Q$.
By Zorn's Lemma, $\calO$ contains a maximal subfield $F$.

\autosyncline{280}Since $F^*$ is contained in $\calO^*$,
it follows that $v$ is trivial on $F$ and that $\pi$ embeds $F$ as a subfield of $k$.
Denote this field by $F^\pi$, we must prove that $F^\pi = k$.
Assume this is not the case and let $\alpha \in k \setminus F^\pi$.

\autosyncline{285}If $\alpha$ is transcendental over $F^\pi$,
choose $\beta \in \calO$ such that $\pi(\beta) = \alpha$.
Since $F[\beta]$ is mapped isomorphically to $F^\pi[\alpha]$,
the valuation $v$ is trivial on $F[\beta]$.
Therefore, it is also trivial on $F(\beta)$,
\autosyncline{290}hence $F(\beta) \subseteq \calO$, contradicting the maximality of $F$.

\autosyncline{292}If $\alpha$ is algebraic over $F^\pi$,
let $\overline{f}(X) \in F^\pi[X]$ be the minimal polynomial of $\alpha$.
Write $f(X)$ for the corresponding polynomial in $F[X]$, under the isomorphism $\pi$.
$\overline{f}(X)$ has a simple zero $\alpha$ in $k$,
so we can use Hensel's Lemma to construct a $\beta \in \calO$ for which $f(\beta) = 0$.
\autosyncline{297}Again, one can prove that $F(\beta) \cong F^\pi(\alpha)$ under $\pi$,
contradicting the maximality of $F$.
\end{proof}

\autosyncline{301}Note that Zorn's Lemma does not imply uniqueness,
so in general this field $F$ is not unique.
Note also that ``$F$ is contained in $\calO$'' is equivalent to
``$v$ is the trivial valuation on $F$'',
so $F$ is maximal with respect to the property that $v$ is trivial on $F$.

\autosyncline{307}In the proof of Proposition~\ref{reshensel},
we only used the hypothesis that $K$ is henselian to exclude
that $k$ is an algebraic extension of $F^\pi$.
So, for non-henselian fields, we can still say the following:
\begin{prop}\label{resnonhensel}
\autosyncline{312}Let $K$ be a valued field with notations as above.
If $\fchar K = \fchar k = 0$,
then $\calO$ contains a maximal subfield $F$.
The projection $\pi$ embeds $F$ as a subfield of $k$,
such that $k$ is algebraic over $\pi(F)$.
\end{prop}

\begin{define}
\autosyncline{321}Let $\Gamma$ be a $\Z$-module.
For a prime $p \in \N$, we say that $\Gamma$ is \emph{$p$-divisible}
if every $x \in \Gamma$ can be written as $p y$, with $y \in \Gamma$.
In other words, if $p \Gamma = \Gamma$.
We call a $\Z$-module \emph{divisible} if it is $p$-divisible for every prime $p$.
\end{define}

\begin{define}
\autosyncline{329}Let $\Gamma$ be a $\Z$-module.
An element $g \in \Gamma$ is called \emph{even} if $g \in 2\Gamma$,
otherwise $g$ is called \emph{odd}.
\end{define}

\autosyncline{334}Clearly, odd elements exist if and only if $\Gamma$ is not $2$-divisible.

\autosyncline{336}We end this section by introducing the \emph{composition} of valuations
(see \cite[Section~2.3, p.~45]{engler-prestel}).
We will only use this in the examples (Section~\ref{sec-examples}).

\begin{prop}\label{compval}
\autosyncline{341}Let $K$ be a field with a valuation $v$ and residue field $k_v$.
Assume $u$ is a valuation on $k_v$, with residue field $k_u$.
Then there exists a valuation $w$ on $K$, called the composition of $v$ with $u$,
with residue field $k_w \cong k_u$ and such that the value groups form an exact sequence
\begin{equation}\label{compval-seq}
	0 \longrightarrow \Gb_u \longrightarrow \Gb_w \longrightarrow \Gb_v \longrightarrow 0
.\end{equation}
\end{prop}

\autosyncline{350}It is easy to prove that $\Gb_w$ is $p$-divisible
if and only if both $\Gb_u$ and $\Gb_v$ are $p$-divisible.
This follows from the exact sequence \eqref{compval-seq},
combined with the fact that the groups are torsion-free.

\subsection{Quadratic forms}

\begin{define}
\autosyncline{360\\funcfield0}A \emph{quadratic form} $Q$ over a field $K$
is a polynomial over $K$ in any number of variables,
which is homogeneous of degree two.

\autosyncline{364}In the case that $\fchar K \neq 2$ (for us this will always be the case),
we can do a linear variable transformation such that $Q$ becomes of the form
$$
	Q(x_1, x_2, \dots, x_n) = a_1 x_1^2 + \dots + a_n x_n^2 \qquad (a_i \in K)
.$$
\autosyncline{369}We abbreviate this as $Q = \form{a_1, \dots, a_n}$.
In what follows, we will always work with quadratic forms in the latter notation.

\autosyncline{372}We define two operators on quadratic forms:
the \emph{orthogonal sum} ($\perp$) and \emph{tensor product} ($\otimes$).
Let $Q_1 = \form{a_1, a_2, \dots, a_n}$ and $Q_2 = \form{b_1, b_2, \dots, b_m}$.
Then
\begin{align*}
	Q_1 \perp Q_2 &= \form{a_1, a_2, \dots, a_n, b_1, b_2, \dots, b_m}, \\
	Q_1 \otimes Q_2 &= \form{a_1b_1, a_1b_2, \dots, a_1b_m, a_2b_1, a_2b_2, \dots, a_2b_m, \dots, a_nb_1, a_nb_2, \dots, a_nb_m}.
\end{align*}
\autosyncline{380}With these operators, the set of quadratic forms over $K$ becomes a semiring.

\autosyncline{382}A quadratic form $\form{a_1, \dots, a_n}$ is called \emph{isotropic} over $K$
if and only if there exist $x_1, \dots, x_n \in K$, not all zero,
such that $a_1 x_1^2 + \dots + a_n x_n^2 = 0$.
Otherwise, the quadratic form is called \emph{anisotropic}.

\autosyncline{387}An important special class of quadratic forms are the \emph{Pfister forms}.
These are the quadratic forms which can be written as
$$
	\form{1, a_1} \otimes \form{1, a_2} \otimes \dots \otimes \form{1, a_n}
.$$
\end{define}

\autosyncline{395}The following propoposition will be crucial to prove the Main Theorem.
It gives a way to reduce isotropicity of quadratic forms
from a valued field $K$ to the residue field $k$,
provided that the value group is not $2$-divisible.

\begin{prop}\label{qfresidue}
\autosyncline{401}Let $K$ be a field with a valuation $v: K^* \onto \Gamma$, and let $k$ be its residue field.
Assume $\fchar k \neq 2$.
Let $t \in K$ have odd valuation (i.e.\ $v(t) \notin 2 \Gamma$).
Consider two quadratic forms
$Q_1 = \form{a_1, \dots, a_n}$ and $Q_2 = \form{b_1, \dots, b_m}$ over $K$,
\autosyncline{406}such that all $a_i$'s and $b_j$'s have valuation $0$.
If $Q_1 \perp (\form{t} \otimes Q_2)$ is isotropic over $K$,
then either $Q_1$ or $Q_2$ is isotropic over the residue field $k$.
\end{prop}

\begin{proof}
\autosyncline{412}For discrete valuations, see \cite[VI.1.9]{lam-qf}.
In the general case, assume $a_1 x_1^2 + \dots + a_n x_n^2 + t b_1 y_1^2 + \dots + t b_m y_m^2 = 0$.
Consider an element from $\{x_1^2, \dots, x_n^2, t y_1^2, \dots, t y_m^2\}$
with minimal valuation.
If $x_i^2$ has minimal valuation, then $a_1 (x_1/x_i)^2 + \dots + a_n (x_n/x_i)^2$
\autosyncline{417}will be zero in the residue field.
If $t y_i^2$ has minimal valuation, then $b_1 (y_1/y_i)^2 + \dots + b_n (y_m/y_i)^2$
will be zero in the residue field.
\end{proof}

\autosyncline{422}If $Q_1 = Q_2$, we can formulate the proposition as follows:
\begin{cor}\label{qfresidue-cor}
Let $K$ be a field with a valuation $v: K^* \onto \Gamma$, and let $k$ be its residue field.
Assume $\fchar k \neq 2$.
Let $t \in K$ have odd valuation.
\autosyncline{427}Consider a quadratic form $Q = \form{a_1, \dots, a_n}$ over $K$,
such that all $a_i$'s have valuation $0$.
If $\form{1,t} \otimes Q$ is isotropic over $K$,
then $Q$ is isotropic over the residue field $k$.
\end{cor}

\autosyncline{433}It is easy to see that
the converse of this proposition and corollary hold for henselian fields:
if $K$ is henselian, and either $Q_1$ or $Q_2$ is isotropic over the residue field,
then $Q_1 \perp (\form{t} \otimes Q_2)$ is isotropic over $K$.

\section{Elliptic curves over function fields}

\autosyncline{442\\funcfield0}Consider an elliptic curve $E$ defined over a field $K$ of characteristic zero.
Such a curve can be defined by an affine equation of the form $Y^2 = f(X) = X^3 + a_2 X^2 + a_4 X + a_6$,
where $f(X)$ has only simple zeros.
There is exactly one point at infinity, which will be denoted by $\ellzero$.
The set of points $E(K)$ forms an abelian group with $\ellzero$ as the neutral element.

\subsection{Denef's method}\label{sec-denef}

\autosyncline{451\\funcfield0}Consider the rational function field $K(Z)$.
Over $K(Z)$ we can define the following quadratic twist of $E$
(sometimes called the \emph{Manin--Denef curve}):
\begin{equation}\label{manin-denef}
	\calE: f(Z)Y^2 = f(X)
.\end{equation}
\autosyncline{457}Consider a point $(X,Y) \in \calE(K(Z))$.
We claim that such a point can be seen as a morphism from $E$ to itself
(morphism as a curve, $\ellzero$ does not have to be mapped to $\ellzero$).
Define the action of $(X,Y) \in \calE(K(Z))$ as follows:
\begin{equation}\label{mdaction}\begin{aligned}
	E(K) &\to E(K) \\
		(x,y) &\mapsto (X(x), Y(x)y)
.\end{aligned}\end{equation}
\autosyncline{465}One can easily check that this is a well-defined morphism on $E(K)$.
The identity is given by $(Z,1)$,
and we denote its multiples $n \cdot (Z,1)$ by $(\chebX_n, \chebY_n) \in \calE(K(Z))$.
This determines rational functions $\chebX_n, \chebY_n \in K(Z)$,
which obviously depend on the elliptic curve $E$.

\autosyncline{471}The curve $\calE$ was first used by Denef
to prove existential undecidability for $\R(Z)$.
The proof is based on the following theorem (see \cite[Lemma~3.1]{denef-real}),
where $\End_K(E)$ stands for the group of endomorphisms of $E$ defined over $K$
and $E[2](K)$ stands for the group of $K$-rational points on $E$ having order dividing $2$.
\begin{theorem}[Denef]\label{denef}
\autosyncline{477}The group $\calE(K(Z))$ is isomorphic to $\End_K(E) \oplus E[2](K)$.
Under this isomorphism, the action \eqref{mdaction}
translates to an action of $(\phi, T) \in \End_K(E) \oplus E[2](K)$ on $E$
by mapping $P \in E(K)$ to $\phi(P) + T$.
\end{theorem}

\autosyncline{483}In our applications,
we will take a curve without complex multiplication (i.e.\ $\End(E) \cong \Z$).
Then $\calE(K(Z)) \cong \Z \oplus E[2](K)$,
hence $2 \cdot \calE(K(Z)) \cong \Z$.
This is how we will make our diophantine model of $\Z$ over $K(Z)$.

\autosyncline{489}It turns out that we can easily describe the functions $\chebX_n$ and $\chebY_n$
locally at $Z\inv$:
\begin{prop}\label{XYinfinity}
Let $n \in \Z \setminus \{0\}$.
In the field $K((Z\inv))$, the functions $\chebX_n$ and $\chebY_n$ satisfy:
$$
	\chebX_n(Z) = \frac{1}{n^2} Z + O(Z^0)
	\quad \text{and} \quad
	\chebY_n(Z) = \frac{1}{n^3} + O(Z^{-1})
.$$
(\autosyncline{499}the notation $f(Z) = g(Z) + O(Z^{-n})$ means $v_{Z\inv}(f - g) \geq n$.)
\end{prop}

\begin{proof}
\autosyncline{503}Apply the following coordinate transformation on $\calE$:
$X = (f(Z)/Z^2) X'$ and $Y = (f(Z)/Z^3) Y'$.
Using $f(X) = X^3 + a_2 X^2 + a_4 X + a_6$, we get
\begin{equation}
	\calE': Y'^2 = X'^3 + a_2 \frac{Z^2}{f(Z)} X'^2 + a_4 \frac{Z^4}{f(Z)^2} X' + a_6 \frac{Z^6}{f(Z)^3}
\end{equation}
\autosyncline{509}Since $f$ has degree $3$, the coefficients of $X'^2$, $X'$ and $1$ in this equation
have positive valuation at $Z\inv$.

\autosyncline{512}Let $P_1'$ be the point $(Z^3/f(Z), Z^3/f(Z))$ on $\calE'$,
this corresponds to $(Z,1)$ on $\calE$.
Let $P_n' := n \cdot P_1'$.
Since $f(Z) = Z^3 + O(Z^2)$, we have to show that $P_n' := (1/n^2 + O(Z\inv), 1/n^3 + O(Z\inv))$.

\autosyncline{517}It suffices to look at the reduction of $\calE'$ modulo $Z\inv$.
The reduction $\overline{\calE'}$ is the cusp $\overline{Y'}^2 = \overline{X'}^3$.
The group law on the set of non-singular points of $\overline{\calE'}(K)$
is isomorphic to the additive group $K,+$ by the following correspondence
(see \cite[III.2.5]{silverman}):
\begin{align*}
	\overline{\calE'}(K) \setminus \{(0,0)\} &\to K,+ \\
	(X',Y') &\mapsto X'/Y'
\end{align*}
\autosyncline{526}Using this, we get $\overline{P_n'} = n \cdot \overline{P_1'} = n \cdot (1,1) = (1/n^2, 1/n^3)$;
hence $P_n' = (n^2 + O(Z\inv), 1/n^3 + O(Z\inv))$.
\end{proof}

\subsection{Moret-Bailly's method}

\autosyncline{534\\funcfield0}In \cite{mb-ell}, Moret-Bailly generalized Denef's method
to make it work for function fields of curves
(and then automatically also higher-dimensional varieties),
as opposed to rational function fields.
The idea is to take an embedding
\autosyncline{539}of $K(Z)$ into a function field $K(C)$ of a curve such that
$\calE(K(Z)) = \calE(K(C))$.

\autosyncline{542}In the theorem below, we will slightly generalize the main theorem by Moret-Bailly.
We warn the reader that this section assumes some familiarity with the paper \cite{mb-ell}.
However, the results are not needed for rational function fields.

\autosyncline{546}We need the definition of \emph{admissible function} from \cite[Definition~1.5.2]{mb-ell}.
We will not need the set $Q$ of closed points on $C$, which does not matter for us.
Essentially, we will ignore condition~(iii), we can simply take any zero
of the admissible function if necessary.
Furthermore, we will always take $\Gamma = E$
(\autosyncline{551}we will however vary the map $\pi: \Gamma \to \PS{1}$).

\begin{define}\label{def-admissible}
\autosyncline{554}Let $C$ be a smooth projective geometrically connected curve over a field $K$ of characteristic zero.
Let $E$ be an elliptic curve over $K$ and $\pi: E \to \PS{1}$ a double cover.
A function $g: C \to \PS{1}$ is called \emph{admissible} for $\pi$ if
\begin{enumerate}
\item[(0)] $\pi: E \to \PS{1}$ is \'etale above $\infty$ and ramified above $0$
(\autosyncline{559}see \cite[1.4.4]{mb-ell}).
\item[(i)] $g$ has no ramification index $\geq 3$ (the ramification is simple).
\item[(ii)] $g$ is \'etale above $\infty$ and the branch points of $\pi$.
\end{enumerate}
\end{define}

\begin{theorem}\label{mb-general}
\autosyncline{566}Let $K$ be a field of characteristic zero and $E$ the elliptic curve $Y^2 = f(X)$,
where $f(X) \in K[X]$ of degree $3$.
Let $C$ be a smooth projective geometrically connected curve defined over $K$.
Let $g: C \to \PS{1}$ be an admissible function for some $\pi: E \to \PS{1}$.
\autosyncline{571}Let $\bar{K}$ denote the algebraic closure of $K$.

\autosyncline{573}Let $\calS$ be a finite set of tuples $(\alpha,\beta,\gamma,\delta,\eps) \in K^5$
such that $f(\alpha) \neq 0$ and $\beta\eps - \gamma\delta \neq 0$.
Then there exist infinitely many $\lambda \in \Q$ such that for every $(\alpha,\beta,\gamma,\delta,\eps) \in \calS$
the set of $\bar{K}(C)$-points of
the elliptic curve $f(\alpha + (\beta\lambda+\gamma)(\delta\lambda+\eps)\inv g\inv) Y^2 = f(X)$
\autosyncline{578}is exactly
$$
	\Z \cdot \left(\alpha + \frac{\beta\lambda+\gamma}{(\delta\lambda+\eps)g}, 1\right) \oplus E[2](\bar{K})
$$
(the $\cdot$ denotes multiplication by an integer on the elliptic curve).
\end{theorem}

\begin{proof}
\autosyncline{586}We need to adapt the proof by Moret-Bailly to account for two things:
first of all, we need several good functions (one for every element of $\calS$).
This works because intersections of Hilbert sets are still Hilbert sets.
Second, we need to some kind of coordinate change
$g\inv \leftrightarrow \alpha + (\beta\lambda+\gamma)(\delta\lambda+\eps)\inv g\inv$.

\autosyncline{592}For every $(\alpha,\beta,\gamma,\delta,\eps) \in \calS$, let $\pi_\alpha$ be the double cover
\begin{align*}
	\pi_\alpha: \hspace{5ex} E &\to \PS{1}: \\
	(X,Y) &\mapsto 1/(X - \alpha)
.\end{align*}
\autosyncline{597}Note that $\pi_\alpha\inv(0)$ is the point $\ellzero$ on $E$
and that $\pi_\alpha\inv(\infty)$ are the points on $E$ with $X$-coordinate $\alpha$.
By assumption, these latter points are not $2$-torsion.
Hence, $\pi_\alpha$ is \'etale over $\infty$ and ramified over $0$.

\autosyncline{602}Let $\calB$ be the union of all the branch points of these $\pi_\alpha$,
excluding $0$.
By assumption, $g$ is admissible for some $\pi: E \to \PS{1}$,
therefore $g: C \to \PS{1}$ is \'etale above an open subset of $\PS{1}$,
which includes $0$ (a branch point of $\pi$) and $\infty$.
\autosyncline{607}It follows that, for almost all $\kappa \in K^*$, the function $\kappa g$ is \'etale
above all points of $\calB$.
Choose such an $\kappa \in \Q^*$.
Then $h := \kappa g$ is admissible for every given $\pi_\alpha$
(note that $g$ and $h$ are equal above $0$ and $\infty$).

\autosyncline{613}Now fix $(\alpha,\beta,\gamma,\delta,\eps) \in \calS$.
Define the following elliptic curves,
depending an a $\xi$ which is an element of some extension of $K$.
\begin{equation}\label{twist}
	\calE_{\alpha,\xi}: f(\alpha + 1/\xi) Y^2 = f(X)
.\end{equation}
\autosyncline{619}If we would strictly follow \cite[1.4.6]{mb-ell},
then we would have the equation $Y^2 = \xi^4 f(\alpha + 1/\xi) f(X)$.
However, the equation \eqref{twist} can be obtained by a coordinate change for the $Y$ variable.

\autosyncline{623}Write $K(Z)$ for the rational function field over $K$.
Note that $K(\alpha + 1/Z) = K(Z)$.
Because $E$ does not have complex multiplication,
Theorem~\ref{denef} says that
\begin{equation}\label{mb-denef}
	\calE_{\alpha,Z}(\bar{K}(Z)) = \Z \cdot (\alpha + 1/Z,1) \oplus E[2](\bar{K})
.\end{equation}

\autosyncline{631}But we want to work over $\bar{K}(C)$ instead of $\bar{K}(Z)$.
The function $h = \kappa g$ is admissible for $\pi_\alpha$, so we can apply \cite[Theorem~1.8]{mb-ell}.
Let $K_0$ be the field generated over $\Q$ by all the coefficients of elements of $\calS$.
There exists a Hilbert subset $\calH_\alpha \subseteq K_0$
such that for all $\mu \in \calH_\alpha$,
\autosyncline{636}we have
\begin{align}\label{mb-isom}
	\calE_{\alpha,Z}(\bar{K}(Z)) &\cong \calE_{\alpha,\mu h}(\bar{K}(C)) \\
	Z &\mapsto \mu h = \kappa \mu g \notag
.\end{align}
(\autosyncline{641}see \cite[Section~11.1]{fried-jarden} for the definition of Hilbert sets,
intuitively a Hilbert set contains `most' elements of $K_0$).
Note that we always have an embedding
$\calE_{\alpha,Z}(\bar{K}(Z)) \into \calE_{\alpha,\mu h}(\bar{K}(C))$,
but in general this is not surjective.

\autosyncline{647}For $(\alpha,\beta,\gamma,\delta,\eps) \in \calS$,
define $\calH'_{(\alpha,\beta,\gamma,\delta,\eps)}$
to be the set of all $\lambda \in K_0$ such that
\begin{equation}\label{mb-transform}
	\frac{\beta \lambda+\gamma}{\delta \lambda+\eps} = \frac{1}{\kappa \mu} \text{ for some } \mu \in \calH_\alpha
.\end{equation}
\autosyncline{653}Since $K_0((\beta Z+\gamma)/(\delta Z+\eps)) = K_0(1/(\kappa Z))$,
if follows from the definition of Hilbert sets that $\calH'_{(\alpha,\beta,\gamma,\delta,\eps)}$
is a Hilbert subset of $K_0$.
Let $\calH'$ be the intersection of all these $\calH'_{(\alpha,\beta,\gamma,\delta,\eps)}$.
Since an intersection of finitely many Hilbert sets is still a Hilbert set
\autosyncline{658}and $K_0$ is finitely generated over the Hilbertian field $\Q$,
it follows that $\calH' \cap \Q$ is infinite.
Now the result follows for all $\lambda \in \calH'$
by putting together \eqref{mb-denef}, \eqref{mb-isom} and \eqref{mb-transform}.
\end{proof}

\section{First version of the Main Theorem}\label{sec-valmain1}

\autosyncline{668\\funcfield0}This whole section is devoted to the proof of the following theorem:

\begin{mtheorem}\label{main-valued1}
\autosyncline{671}Let $K$ be a field of characteristic zero with a valuation $v: K^* \onto \Gamma$.
Let $\calO$ denote the valuation ring and $k$ the residue field.

\autosyncline{674}Assume the following conditions are satisfied:
\begin{enumerate}[(i)]
\item\label{cond1-char} The characteristic of the residue field $k$ is zero.
\item\label{cond1-gamma} The value group $\Gamma$ is not $2$-divisible.
\item\label{cond1-qf}
	\autosyncline{679}Let $F$ be a maximal field contained in $\calO$.
	There is an integer $q \geq 0$ such that there exists
	a $2^q$-dimensional Pfister form with coefficients in $F$
	which is anisotropic over $k$
	and such that every $2^{q+2}$-dimensional Pfister form
	\autosyncline{684}over a finite extension of $F(Z)$ is isotropic.
\end{enumerate}

\autosyncline{687}Let $C$ be a smooth projective geometrically connected curve defined over $K$
with a $K$-rational point.
Let $K(C)$ be its function field.
Then there exists a diophantine model of $\Z$ over $K(C)$
with coefficients in some finitely generated subfield $\calL_0$ of $K(C)$.
\end{mtheorem}

\begin{remark}
\autosyncline{695}This implies the negative answer to HTP for $K(C)$ with coefficients in $\calL_0$
by Proposition~\ref{diophmodel}.
However, as Eisentr\"ager notes in the introduction of \cite{eisentraeger-padic},
this undecidability can be ``trivial'' in some cases,
simply because of certain elements appearing in $\calL_0$.
\autosyncline{700}To explain this better,
consider Tarski's proof that the theory of $\R$ in the language $\{0,1,+,\cdot,\leq\}$
admits quantifier elimination (see \cite{tarski-real}).
This immediately implies decidability for first-order sentences
(in particular, diophantine equations).
\autosyncline{705}However, if we add some non-computable real $\alpha$ to the language,
we still have quantifier elimination,
but then atomic formulas (such as $2\alpha^3 - \alpha + 4 \geq 0$) are no longer decidable.
This shows that undecidability can sometimes be a simple consequence
of the language.

\autosyncline{711}However, for a general field $K$, it is not at all clear what the natural language
(or the corresponding field $\calL_0$) should be.
In Section~\ref{sec-vallang}, we will discuss the coefficient field $\calL_0$.
In the concrete examples in Section~\ref{sec-examples},
we will see that this field $\calL_0$ is the natural one which one would expect.
\end{remark}

\autosyncline{718}To prove the Main Theorem,
we would like to use the method with two elliptic curves,
as applied on $\C(T, Z)$ by Kim and Roush (\cite{kim-roush-ct1t2})
and on function fields of surfaces over $\C$ by Eisentr\"ager (\cite{eisentraeger-funcfield0}).
The big obstacle however is that $K$ might be much bigger than $F(T)$;
\autosyncline{723}it could be that there is no rank one elliptic curve over $K$.

\autosyncline{725}Take an element $T \in K$ such that $v(T)$ is positive and odd
(this is possible because of condition (\ref{cond1-gamma})).
We will identify $\Z$ with a subgroup of $\Gamma$ by sending $1$ to $v(T)$.
An ordered $\Z$-module is always torsion-free,
so the map $\Z \into \Gamma: n \mapsto n v(T)$ is an embedding of ordered $\Z$-modules.

\subsection{The elliptic curve}\label{sec-ellcurve}

\autosyncline{735\\funcfield0}Let $E$ be an elliptic curve over $\Q$ without complex multiplication.
Choose an equation $Y^2 = f(X) = X^3 + a_2 X^2 + a_4 X + a_6$
for $E$ with $a_2, a_4, a_6 \in \Q$ and $a_6 \neq 0$.
Let $\pi: E \to \PS{1}: (X,Y) \mapsto X\inv$.
Since $C$ was assumed to have a $K$-rational point,
\autosyncline{740}it follows from \cite[2.3.3]{mb-ell} that there exists a
function $g: C \to \PS{1}$ of odd degree which is admissible for $\pi$.
Define $Z$ to be $g\inv$.
In what follows, we will see $Z$ as an element of the function field $K(C)$.
Then $K(C)$ is a finite extension of odd degree of the rational function field $K(Z)$.

\autosyncline{746}Apply Theorem~\ref{mb-general} with
$\calS = \{(0,1,T\iinv,0,1), (T\iinv,1,0,0,1)\}$
and let $\lambda \in \Q^*$ be such that the conclusion of that theorem holds.
Define
$$
	A := (T\iinv + \lambda) Z \quad \text{and} \quad B := T\iinv + \lambda Z
.$$
\autosyncline{753}In the case of a rational function field,
we can take $K(Z) = K(C)$ and then any $\lambda \in \Q^*$ will work.

\autosyncline{756}Define $L := K(C)(\sqrt{f(A)}, \sqrt{f(B)})$,
which will turn out to be a degree $4$ extension of $K(C)$.
In what follows, we assume that we have $T$ and $Z$ in the field of coefficients $\calL_0$.
Both $A$ and $B$ are elements of $\Q(T,Z)$ and $f$ has coefficients in $\Q$,
therefore $f(A)$ and $f(B)$ are diophantine and we can make a diophantine model of $L$ in $K(C)^4$.

\autosyncline{762}Consider the following points on $E(L)$:
$$
	\textstyle
	P_1 := (A, \sqrt{f(A)})
	\quad \text{and} \quad
	P_2 := (B, \sqrt{f(B)})
$$

\begin{lemma}\label{elllemma}
\autosyncline{771}The points $P_1$ and $P_2$ satisfy the following properties:
\begin{enumerate}
\item\label{elllemma-dio} Let $\Z_0 = \Z \setminus \{0\}$.
	The sets of multiples $\Z_0 \cdot P_1$ and $\Z_0 \cdot P_2$ are diophantine over $L$
	(as subsets of $\mathbb{A}^2(L) \cong L^2$).
\item\label{elllemma-indep} $P_1$ \autosyncline{776}and $P_2$ are independent points on $E(L)$.
\item\label{elllemma-deg} Let $\bar{K}$ be the algebraic closure of $K$.
	Then the field $\bar{K}(C)(\sqrt{f(A)}, \sqrt{f(B)})$
	is a degree $4$ extension of $\bar{K}(C)$.
\end{enumerate}
\end{lemma}

\begin{proof}
\autosyncline{784}Let $\calE_A$ be the elliptic curve $f(A) Y^2 = f(X)$
and $\calE_B$ be the elliptic curve $f(B) Y^2 = f(X)$,
both defined over $K(C)$.
According to Theorem~\ref{mb-general},
we have $\calE_A(K(C)) = \Z \cdot (A,1) \oplus E[2](K)$
\autosyncline{789}and $\calE_B(K(C)) = \Z \cdot (B,1) \oplus E[2](K)$.

\autosyncline{791}The set of multiples of $(A,1)$ on $\calE_A(K(C))$ is diophantine because it can be written as
$$
	\big\{2 \cdot \calE_A(K(C))\big\} \cup \big\{(A,1) + 2 \cdot \calE_A(K(C))\big\}
.$$
Since the $K(C)$-rational points of $\calE_A$ are simply given by the elliptic curve equation,
\autosyncline{796}the above set is diophantine.
We will use the affine equation, so we cannot get the point at infinity,
we only get $\Z_0 \cdot (A,1)$.
The coefficients of the equation for $\calE_A$ lie in $\Q(T,Z)$,
so we just need $T$ and $Z$ in $\calL_0$ to make the diophantine definition.

\autosyncline{802}Over $L = K(C)(\sqrt{f(A)}, \sqrt{f(B)})$, the curves $\calE_A$ and $E$ become isomorphic:
\begin{equation}\label{elllemma-isom}\begin{aligned}
	\theta: \calE_A(L) &\simto E(L) \\
	(X,Y) &\mapsto (X,Y\sqrt{f(A)})
.\end{aligned}\end{equation}
\autosyncline{807}Now we can diophantinely define the set of non-zero multiples of
$P_1 = (A, \sqrt{f(A)})$ on $E(L)$ by taking the multiples of $(A,1)$ on $\calE_A(L)$
and simply multiplying the $y$-coordinate by $\sqrt{f(A)}$.
Analogously, the set $\Z_0 \cdot P_2$ is diophantine,
which finishes the first point of the lemma.

\autosyncline{813}To prove \ref{elllemma-indep},
first of all note that both $P_1$ and $P_2$ have infinite order in $E(L)$
because of Theorem~\ref{denef}.
Assume we would have a relation $mP_1 = nP_2$ with $m \neq 0$ and $n \neq 0$.
Since the $x$-coordinate of $P_1$ is $A$,
\autosyncline{818}it follows from Section~\ref{sec-denef} that the $x$-coordinate of $mP_1$ equals $\chebX_m(A)$.
Similarly, the $x$-coordinate of $nP_2$ is $\chebX_n(B)$.
So, we have $\chebX_m((T\iinv + \lambda)Z) = \chebX_n(T\iinv + \lambda Z)$.
If we specialize the variable $Z$ to $T\inv$,
we get $\chebX_m(T^{-3} + \lambda T\inv) = \chebX_n(T^{-2} + \lambda T\inv)$.
\autosyncline{823}But it follows from Proposition~\ref{XYinfinity} that
$v(\chebX_m(T^{-3} + \lambda T\inv)) = -3$
and $v(\chebX_n(T^{-2} + \lambda T\inv)) = -2$.
This is a contradiction.

\autosyncline{828}Finally, let us prove point \ref{elllemma-deg}.
Assume that $\sqrt{f(A)}$ is in $\bar{K}(C)$.
Then the isomorphism $\theta$ in \eqref{elllemma-isom} would be defined over $K(C)$.
Since $E(\bar{K})$ contains $n$-torsion points for every $n$,
$\calE_A(\bar{K}(C))$ would also contain $n$-torsion points for every $n$.
\autosyncline{833}But by our construction, $\calE_A(\bar{K}(C))$
has only $2$-torsion points and points of infinite order.
Therefore, $[\bar{K}(C)(\sqrt{f(A)}):\bar{K}(C)] = 2$.

\autosyncline{837}Now assume that $\sqrt{f(B)} \in \bar{K}(C)(\sqrt{f(A)})$.
Then we can write $\sqrt{f(B)} = R + S\sqrt{f(A)}$ with $R$ and $S$ in $\bar{K}(C)$.
Squared, we get
$$
	f(B) = R^2 + S^2f(A) + 2RS \sqrt{f(A)} \in \bar{K}(C)
.$$
\autosyncline{843}But $\sqrt{f(A)}$ does not lie in $\bar{K}(C)$,
so we have two possibilities: either $R = 0$ or $S = 0$.
If $S = 0$, then $\sqrt{f(B)} \in \bar{K}(C)$, which we can exclude as in the previous paragraph.

\autosyncline{847}If $R = 0$, then $\sqrt{f(B)}$ is a $\bar{K}(C)$-multiple of $\sqrt{f(A)}$.
Then $(B, \sqrt{f(B)}/\sqrt{f(A)})$ would be a point on $\calE_A(\bar{K}(C))$.
This means that $2$ times this point is a multiple of $(A, 1)$.
Applying the isomorphism $\theta$,
we find that $2 \cdot P_2$ is a multiple of $P_1$,
\autosyncline{852}in contradiction with the independence of $P_1$ and $P_2$.
\end{proof}

\autosyncline{856}We have to make a technical remark about affine versus projective points.
We just defined $\Z_0 \cdot P_i$, the affine multiples of $P_i$.
However, we would also like to work with the point at infinity.
So we work with projective coordinates in $\PS[L]{2} = (L^3 \setminus \{0\})/L^*$.
The equivalence relation between different coordinates for the same point is clearly diophantine.
\autosyncline{861}Now $\Z \cdot P_i$ = $(0,1,0) \cup \set{(X,Y,1)}{(X,Y) \in \Z_0 \cdot P_i}$.

\subsection{The model of $\Z \times \Z$}\label{sec-ZxZ}

\autosyncline{867\\funcfield0}Consider the set $\Z \times \Z$ with the obvious addition $(a,b) + (c,d) = (a+c, b+d)$.
Define a binary relation $\mid$ on $\Z \times \Z$ which satisfies
\begin{equation}\label{ZxZ-mid}
	a \text{ odd }
		\implies
	(a,1) \mid (c,d)
		\leftrightarrow
	(\exists r \in \Z) \big((c,d) = r(a,1)\big)
		\leftrightarrow
	c = a d
.\end{equation}
\autosyncline{878}Note that the truth of $(a,b) \mid (c,d)$ with $a$ even or with $b \neq 1$ does not matter,
we can define $\mid$ as we wish for such arguments.

\autosyncline{881}If we embed $\Z$ into $\Z \times \Z$ by mapping $n$ to $(n,0)$,
then we can existentially define the addition and multiplication
on the image of $\Z$ in terms of the relations $+$ and $\mid$ on $\Z \times \Z$.
For the addition, this is obvious since $a + b = c$ is equivalent to $(a,0) + (b,0) = (c,0)$.
For the multiplication, we have:
\begin{prop}\label{ZxZmul}
\autosyncline{887}Let $a,b,c \in \Z$.
Then $a b = c$ if and only if there exists an $X \in \Z \times \Z$
such that the following relations are satisfied:
\begin{align}
	(1,1) &\mid X \label{ZxZmul-plus} \\
	(-1,1) &\mid (X - 2(b,0)) \label{ZxZmul-minus} \\
	(2(a,0) + (1,1)) &\mid (X + 2(c,0)) \label{ZxZmul-c}
\end{align}
\end{prop}

\begin{proof}
\autosyncline{898}First of all, note that these 3 relations are all of the form \eqref{ZxZ-mid}.
If $a b = c$, then $X = (b,b)$ satisfies the relations.
Conversely, if the relations are satisfied,
then $X$ must be of the form $(x,x)$ by \eqref{ZxZmul-plus}.
Now \eqref{ZxZmul-minus} says that $(-1,1) \mid (x-2b, x)$.
\autosyncline{903}This implies that $b = x$.
Using $X = (b,b)$, equation \eqref{ZxZmul-c}
becomes $(2a+1, 1) \mid (b+2c, b)$ which implies $b+2c = (2a+1)b$, hence $c = a b$.
\end{proof}

\autosyncline{908}We will apply this as follows:
as shown in section~\ref{sec-ellcurve}, we can diophantinely define the sets
$\Z \cdot P_1$ and $\Z \cdot P_2$, hence also $\Z \cdot P_1 + \Z \cdot P_2$
inside $E(L) \subseteq \PS[L]{2}$.

\autosyncline{913}We identify $\Z \cdot P_1 + \Z \cdot P_2$ with $\Z \times \Z$ via
$a P_1 + b P_2 \longleftrightarrow (a,b) \in \Z \times \Z$.
Then the addition on $\Z \times \Z$ corresponds to addition on the elliptic curve,
so it is diophantine.
In section~\ref{sec-divqf} we will show that also the relation $\mid$ is diophantine.
\autosyncline{918}This would show that the map $\Z \to E(L): n \mapsto n P_1$ is a diophantine model of $\Z$,
which is what we were asked to proved in Main Theorem~\ref{main-valued1}.

\subsection{The quadratic form}\label{sec-divqf}

\autosyncline{925\\funcfield0}The following theorem shows that the relation $\mid$ (see \eqref{ZxZ-mid})
on $\Z \cdot P_1 + \Z \cdot P_2 \cong \Z \times \Z$ is diophantine.

\begin{theorem}\label{divqf}
\autosyncline{929}Let $Q$ be a $2^q$-dimensional anisotropic Pfister form over $k$ with coefficients in $F$,
which exists by assumption.
Let $m,n,r \in \Z$ with $m$ odd.
Then $n = mr$ if and only if $nP_1 + rP_2 = \ellzero$ or
\begin{equation}\label{divqf-qf}
	\form{1, y(m P_1 + P_2)} \otimes \form{1, y(n P_1 + r P_2)} \otimes Q
\end{equation}
\autosyncline{936}is isotropic over $L$.
($y(P)$ stands for the $y$-coordinate of the point $P$.)
\end{theorem}

\begin{remark}
\autosyncline{941}A quadratic form being isotropic is a diophantine condition
if all the coefficients are diophantine.
Therefore, the coefficients of $Q$ must be elements of the field of coefficients $\calL_0$.
\end{remark}

\begin{proof}
\autosyncline{947}The statement clearly holds if $n = r = 0$.
For the rest of the proof, we assume this is not the case.

\autosyncline{950}Assume $n = mr$ and set $P_3 := m P_1 + P_2$.
Now \eqref{divqf-qf} becomes
\begin{equation}\label{qf2}
	\form{1, y(P_3)} \otimes \form{1, y(r P_3)} \otimes Q
.\end{equation}
\autosyncline{955}Since $y(r P_3) = \chebY_r(x(P_3)) y(P_3)$,
the coefficients of this quadratic form live in $L_0 := F(x(P_3), y(P_3))$.
This field is isomorphic to the function field of $E$ over $F$,
so we can use condition \eqref{cond1-qf} from the Theorem.
The Pfister form \eqref{qf2} is $2^{q+2}$-dimensional,
\autosyncline{960}therefore it is isotropic over $L_0 \subseteq L$.

\autosyncline{962}Conversely, assume that \eqref{divqf-qf} is isotropic over $L$.
Let $s := n - mr$ and suppose that $s \neq 0$ in order to find a contradiction.
Putting $P_3 := m P_1 + P_2$, we rewrite \eqref{divqf-qf} as
\begin{equation}\label{qf3}
	\form{1, y(P_3)} \otimes \form{1, y(s P_1 + r P_3)} \otimes Q
.\end{equation}

\autosyncline{969}For the rest of this proof, we take the henselisation $K\hens$ as a base field, instead of $K$.
Take any extension of the valuation $v$ to $K\hens$.
By abuse of notation, we will still write $v$ for this valuation.
This extension is immediate,
which means that the value group $\Gamma$ and the residue field $k$ remain the same.
\autosyncline{974}The henselisation is an algebraic extension,
and $K$ is relatively algebraically closed in $L$
(because $K(C)$ is a function field over $C$ and because of Lemma~\ref{elllemma}, item \ref{elllemma-deg}).
Define
$$
	M := L \otimes_K K\hens = K\hens(C)(\sqrt{f(A)}, \sqrt{f(B)})
.$$
\autosyncline{981}Since \eqref{qf3} is isotropic over $L$, it is certainly isotropic over $M$.
We just need the field $M$ for this proof,
we certainly do not need a diophantine model of $M$.

\autosyncline{985}Recall that $m$ is odd, in particular $m$ is non-zero.
The points $mP_1$ and $P_2$ have the following coordinates:
\begin{align}
	mP_1 &= \big(\chebX_m(A), \chebY_m(A)\sqrt{f(A)}\big),\\
	P_2 &= (B, \sqrt{f(B)}).
\end{align}

\autosyncline{992}Consider $H(Z) := \chebX_m(A) - B \in K\hens(Z)$,
we want to find a simple zero $\gamma \in K\hens$ of this rational function.
Write the rational function $\chebX_m(\xi) \in \Q(\xi)$ as $R_m(\xi)/S_m(\xi)$
with $R_m(\xi), S_m(\xi) \in \Q[\xi]$.
By Proposition~\ref{XYinfinity},
\autosyncline{997}we can choose these such that $R_m$ has leading term $\xi^d$
and $S_m$ has leading term $m^2 \xi^{d-1}$ for some $d$.
Keeping in mind that $A = (T\iinv + \lambda) Z$ and $B = T\iinv + \lambda Z$
with $\lambda \in \Q^*$,
the following is a polynomial in $Z$ with coefficients in $\Q[T] \subseteq \calO$:
\begin{equation}\label{g}
	G(Z) := T^{2d} S_m(A) H(Z) = T^{2d} R_m(A) - T^{2d} S_m(A) B
.\end{equation}

\autosyncline{1006}We would like to apply Hensel's Lemma to find a root of $G(Z)$ in $K\hens$.
Modulo $\frakm$ (the maximal ideal in $\calO$ containing $T$), we have
\begin{align*}
	G(Z) &\equiv (T^2 A)^d - m^2 (T^2 A)^{d-1} (T^2 B) \mod{\frakm} \\
		&\equiv Z^d - m^2 Z^{d-1} \mod{\frakm}
.\end{align*}
\autosyncline{1012}This equation has a simple zero $m^2 \mod{\frakm}$,
therefore Hensel's Lemma shows that
$G(Z)$ has a simple root $\gamma \in K\hens$ with $\gamma \equiv m^2 \mod{\frakm}$.

\autosyncline{1016}In order for $\gamma$ to be a zero of the rational function $H(Z) = T^{-2d} G(Z)/S_m(A)$,
it must not be a zero of $S_m(A) = S_m((T\iinv + \lambda) Z)$.
But $S_m$ has coefficients in $\Q$,
so the zeros of $S_m((T\iinv + \lambda) Z)$ are of the form
$\alpha/(T\iinv + \lambda)$ with $\alpha$ algebraic over $\Q$.
\autosyncline{1021}Since $\gamma$ has valuation zero, it is clearly not of this form.

\autosyncline{1023}Define $w$ as the discrete valuation on $K\hens(Z)$ at the point $Z = \gamma$.
This means that $w(Z - \gamma) = 1$ and that $w$ is trivial on $K\hens$.
Clearly, the residue field is $K\hens$.
We found $\gamma$ as a simple zero of $H(Z) = \chebX_m(A) - B$, therefore
\begin{equation}\label{xequiv}
	w(\chebX_m(A) - B) = 1
.\end{equation}

\autosyncline{1031}We defined $w$ as a valuation on $K\hens(Z)$,
but we would like to extend $w$ to the finite extension $M = K\hens(C)(\sqrt{f(A)}, \sqrt{f(B)})$.
We use the notation $\tilde{x}$ for the reduction of $x$ with respect to $w$,
this gives a map $K\hens[Z]_{(Z-\gamma)} \rightarrow K\hens$.
As we extend $w$ to a finite extension, we keep the same notation.

\autosyncline{1037}Since $[K\hens(C) : K\hens(Z)]$ is odd, it follows from Proposition~\ref{valuation-ext}
that we can extend $w$ to $K\hens(C)$ in such a way
that both the ramification index $e_w$ and the residue extension degree $f_w$ are odd.
Choose such an extension and write ${K\hens}'$ for the residue field
of this extended valuation.
\autosyncline{1042}The new value group is generated by $1/e_w$, we do not renormalize.
Since algebraic extensions of henselian fields are again henselian
(see \cite[Section~4.1]{engler-prestel}),
${K\hens}'$ is also henselian (for the extension of $v$ to ${K\hens}'$).

\autosyncline{1047}Now we still have to adjoin $\sqrt{f(A)}$ and $\sqrt{f(B)}$ to $K\hens(C)$.
Note that $\tilde{A} = (T\iinv + \lambda) \gamma$
and $\tilde{B} = T\iinv + \lambda \gamma$ with $\lambda \in \Q^*$
and $\gamma \equiv m^2 \mod{\frakm}$.
It follows that $T^6 f(\tilde{A}) \equiv m^6 \mod{\frakm}$
\autosyncline{1052}and also $T^6 f(\tilde{B}) \equiv 1 \mod{\frakm}$.
Hensel's Lemma implies that
$f(\tilde{A})$ and $f(\tilde{B})$ are squares in ${K\hens}'$.
After extending $w$ to $M = K\hens(C)(\sqrt{f(A)},\sqrt{f(B)})$,
the residue field remains ${K\hens}'$
\autosyncline{1057}and $w$ does not ramify in this extension $M/K\hens(C)$.

\autosyncline{1059}Equation \eqref{xequiv} implies that $m \tilde{P_1}$ and $\tilde{P_2}$
have the same $x$-coordinate (an element of $K\hens$).
This means that there are $2$ possibilities:
either they are the same point (equal $y$-coordinates),
or they are opposite points (opposite $y$-coordinates).
\autosyncline{1064}But $M$ has an involution $\sigma$ mapping $\sqrt{f(B)}$ to $-\sqrt{f(B)}$,
while fixing $K\hens(C)(\sqrt{f(A)})$ (this follows from Lemma \ref{elllemma}).
On the curve, $\sigma(P_1) = P_1$ but $\sigma(P_2) = -P_2$.
We want $m \tilde{P_1}$ and $\tilde{P_2}$ to be opposite points.
If this is not the case, replace $w$ by the valuation $w \circ \sigma$.
\autosyncline{1069}Then the points become opposite and
\begin{equation}\label{yequiv}
	w\left(\chebY_{m}(A) \sqrt{f(A)} - \sqrt{f(B)}\right) = 0
.\end{equation}

\autosyncline{1074}We will now determine $w(y(P_3))$ using the fact that $P_3 = mP_1 + P_2$.
We can do this with \eqref{xequiv} and \eqref{yequiv}.
The elliptic curve addition formula says that
\begin{align*}
	x(P_3) &= -a_2 -x(mP_1) -x(P_2) + \left(\frac{y(mP_1) - y(P_2)}{x(mP_1) - x(P_2)}\right)^{\!\!2} \\
		&= -\underbrace{a_2}_{w \geq 0} -\underbrace{\chebX_m(A)}_{w = 0}
			- \underbrace{B}_{w = 0}
			+ \underbrace{\left(\frac{\chebY_m(A) \sqrt{f(A)} - \sqrt{f(B)}}{\chebX_m(A) - B}\right)^{\!\!2}}_{w = 2(0 - 1) = -2}
.\end{align*}
\autosyncline{1083}We see that $w(x(P_3)) = -2$.
The elliptic curve equation $y(P_3)^2 = f(x(P_3))$ implies that $w(y(P_3)) = -3$.
This should indeed be negative because we already knew that $\tilde{P_3}$ is the point at infinity.

\autosyncline{1087}So far we determined the $w$-valuation of the coefficient $y(P_3)$ in the quadratic form \eqref{qf3}.
We claim that $w(y(s P_1 + r P_3)) = 0$.
If $w(y(s P_1 + r P_3)) < 0$,
then $s \tilde{P_1} + r \tilde{P_3} = s \tilde{P_1} = \ellzero$;
if $w(y(s P_1 + r P_3)) > 0$,
\autosyncline{1092}then the $y$-coordinate of $s \tilde{P_1} + r \tilde{P_3} = s \tilde{P_1}$ is zero,
hence $s \tilde{P_1}$ is $2$-torsion.
In any case, if $w(y(s P_1 + r P_3)) \neq 0$,
then $\tilde{P_1}$ is a torsion point on $E$ (here we need $s \neq 0$).
But $E$ has coefficients in $\Q$, hence all torsion is algebraic over $\Q$.
\autosyncline{1097}The $x$-coordinate of $\tilde{P_1}$ is $\tilde{A} = (T\iinv + \lambda)\gamma$ with $v(\tilde{A}) = -2$,
therefore $\tilde{A}$ cannot be algebraic over $\Q$
and $\tilde{P_1}$ cannot be torsion.

\autosyncline{1101}We conclude $w(y(P_3)) = -3$ and $w(y(s P_1 + r P_3)) = 0$.
We would like to apply Corollary~\ref{qfresidue-cor} on \eqref{qf3}.
This works because $-3$ is odd in the value group of $w$;
indeed the value group is $(1/e_w) \Z$ with $e_w$ odd.
So Corollary~\ref{qfresidue-cor} gives us that
\begin{equation}\label{qfl}
	\form{1, y(s \tilde{P_1})} \otimes Q
.\end{equation}
\autosyncline{1109}is isotropic over ${K\hens}'$.

\autosyncline{1111}Recall that $[{K\hens}':K\hens] = f_w$ is odd.
Since $K\hens$ is henselian,
the valuation $v$ on $K\hens$ can be extended to ${K\hens}'$
in a unique way.
This extension has ramification index $e_v$ and residue extension degree $f_v$
\autosyncline{1116}which satisfy $e_v f_v = [{K\hens}':K\hens] = f_w$,
therefore both $e_v$ and $f_v$ are odd.
Write $k'$ for the new residue field.

\autosyncline{1120}The point $\tilde{P_1}$ has $x$-coordinate $\tilde{A} = (T\iinv + \lambda)\gamma$
with $v(\tilde{A}) = -2$.
The $y$-coordinate of $s \tilde{P_1}$ equals $\chebY_s(\tilde{A}) \sqrt{f(\tilde{A})}$.
Proposition~\ref{XYinfinity} implies that $v(\chebY_s(\tilde{A})) = 0$;
hence $v(y(s \tilde{P_1})) = v(f(\tilde{A}))/2 = -3$.
\autosyncline{1125\\funcfield0}Since $e_v$ is odd, a similar reasoning as before implies that this $-3$ is an odd element
of the value group of $v$ on ${K\hens}'$.
We can apply Corollary~\ref{qfresidue-cor} on \eqref{qfl} to conclude
that $Q$ is isotropic over the residue field $k'$.
Since $[k':k] = f_v$ is odd and $Q$ has coefficients in $F \subseteq k$,
\autosyncline{1130}it follows from Springer's Theorem (see \cite[VII.2.7]{lam-qf}) that $Q$ is also isotropic over $k$.
But $Q$ was chosen to be anisotropic over $k$, so we have found a contradiction.
\end{proof}

\section{The conditions of the Main Theorem}\label{sec-valmain2}

\autosyncline{1138\\funcfield0}It turns out that we can simplify some of the conditions of Main Theorem~\ref{main-valued1}.
First of all, thanks to Voevodsky's work on the Milnor Conjectures
(see \cite{pfister-milnor} for a survey),
we can replace condition \eqref{cond1-qf} in Main Theorem~\ref{main-valued1}
by a simple condition on the $2$-cohomological dimensions
\autosyncline{1143}of $\gal(\bar{F}/F)$ and $\gal(\bar{k}/k)$.
Second, the condition that the curve $C$ has a rational point
can be easily removed by going to a finite extension of $K$.

\subsection{Galois Cohomology}\label{seccohom}

\autosyncline{1151\\funcfield0}We will recall some definitions and propositions from Galois cohomology,
we refer to \cite{serre-cohom} for background and proofs.

\autosyncline{1154}Throughout this section, $K$ will be a characteristic zero field.
Let $H^q(K, \mu_p)$ denote the $q$-th cohomology group
of the absolute Galois group $\gal(\bar{K}/K)$
with coefficients in the group $\mu_p \subset \bar{K}^*$ of $p$-th roots of unity.

\begin{define}
\autosyncline{1160}Let $p$ be a prime number.
The \emph{$p$-cohomological dimension} of $\gal(\bar{K}/K)$, denoted by $\cdim_p(K)$,
is the smallest integer $q$ such that
$$
	H^{q+1}(L, \mu_p) = 0 \qquad \text{for all finite extensions $L$ of $K$}
.$$
\autosyncline{1166}If there is no such $q$, then we define $\cdim_p(K) = \infty$.
\end{define}

\autosyncline{1169}Serre gives a different definition of $p$-cohomological dimension,
but ours is equivalent,
see the proof of \cite[II.\S\,2.3~Prop.~4]{serre-cohom}.

\autosyncline{1173}It turns out that we can describe
how these cohomological dimensions behave with respect to field extensions:
\begin{prop}[see {\cite[II.\S\,4.2~Prop.~11]{serre-cohom}}]\label{cohom-ext}
Let $K$ be a characteristic zero field with $\cdim_p(K) < \infty$,
and let $L$ be any extension of $K$.
\autosyncline{1178}Then
\begin{equation}
	\cdim_p(L) \leq \cdim_p(K) + \mathrm{tr.\:deg}(L/K)
.\end{equation}
If $L$ is finitely generated over $K$, the equality holds.
\autosyncline{1183}In particular, cohomological dimensions remain the same under finite extensions,
provided that $\cdim_p(K) < \infty$.
\end{prop}

\autosyncline{1187}The Milnor Conjectures, now proven by Voevodsky and others,
provide a connection between the Witt ring $W(K)$
(an object used to study quadratic forms, see for example \cite[Chapter~II]{lam-qf})
and the Galois cohomology groups $H^q(K,\mu_2)$:
\begin{theorem}\label{milnor}
\autosyncline{1192}Let $I$ denote the fundamental ideal (generated by the $2$-dimensional forms) in $W(K)$.
Then $I^q/I^{q+1} \cong H^q(K, \mu_2)$.
\end{theorem}

\autosyncline{1196}Using this, we know the possible dimensions of anisotropic Pfister forms over $K$:
\begin{cor}\label{milnor-cor}
There exists an an\-i\-so\-tro\-pic $2^q$-dimensional Pfister form over $K$
if and only if $H^q(K, \mu_2) \neq 0$.
\end{cor}

\begin{proof}
\autosyncline{1203}If $H^q(K, \mu_2) = 0$, then $I^q/I^{q+1} = 0$.
This implies that $I^q = I^{q+1}$, hence also $I^{q+1} = I^{q+2}$ and so on.
The Arason--Pfister Hauptsatz (see \cite[X.5.1]{lam-qf})
implies that $\bigcap_{n \geq 0} I^n = 0$, therefore $I^q = 0$.
But $I^q$ is generated by the $2^q$-dimensional Pfister forms,
\autosyncline{1208}therefore all $2^q$-dimensional Pfister forms are hyperbolic (hence isotropic).

\autosyncline{1210}Conversely, if $H^q(K, \mu_2) \neq 0$, then $I^q \neq 0$.
Therefore, there exists a non-hyperbolic Pfister form $Q$ of dimension $2^q$.
But for Pfister forms, non-hyperbolic is the same as anisotropic.
\end{proof}

\autosyncline{1216}We can now change condition \eqref{cond1-qf} from Main Theorem~\ref{main-valued1}:

\begin{prop}\label{wlog-qf}
\autosyncline{1219}Main Theorem~\ref{main-valued1} is still true if we replace condition \eqref{cond1-qf} by:
``the $2$-cohomological dimensions of $F$ and $k$ are equal and finite.''
We can do this without loss of generality.
\end{prop}

\autosyncline{1224}Note that this does \emph{not} mean that condition \eqref{cond1-qf} from the Main Theorem
is equivalent to ``$\cdim_2(F) = \cdim_2(k) < \infty$'',
it just means that we can also prove the Main Theorem
with the new condition instead of \eqref{cond1-qf}.
When we say ``without loss of generality'',
\autosyncline{1229}it means that ``$\cdim_2(F) = \cdim_2(k) < \infty$''
always holds if \eqref{cond1-qf} is satisfied.
We might need to extend the field $\calL_0$ though.

\begin{proof}
\autosyncline{1234}Assume $q := \cdim_2(F) = \cdim_2(k)$ is finite
and that conditions~(\ref{cond1-char}) and (\ref{cond1-gamma}) are satisfied.
By definition of cohomological dimension,
there is a finite extension $k_1/k$ for which $H^q(k_1, \mu_2) \neq 0$.

\begin{figure}
$$\xymatrix@!C@!R{
	& & K' = K_1(\beta) \\
	& K_1 = K(\alpha) \ar@{-}[ur] & \calO' \ar@{-}[u] \ar@{->>}[r] & k' \\
	K \ar@{-}[ur] & \calO_1 \ar@{-}[u] \ar@{-}[ur] \ar@{->>}[r] & k_1 \ar@{=}[ur] & F' \ar@(l,d)@{-}[ul]|(.583)\hole \ar@{-}[u] \\
	\calO \ar@{-}[u] \ar@{-}[ur] \ar@{->>}[r] & k \ar@{-}[ur] & F_1 = F(\beta) \ar@(l,d)@{-}[ul]|(.583)\hole \ar@{-}[u] \ar@{-}[ur] \\
	& F \ar@(l,d)@{-}[ul] \ar@{-}[u] \ar@{-}[ur]
}$$
\end{figure}

\autosyncline{1249}By \cite[Theorem~(27.1)]{endler},
we can find an extension $K_1/K$ such that $v$ extended to $K_1$
has residue field $k_1$ and value group $\Gamma$.
Choose $\alpha$ in the algebraic closure $\bar{K}$
such that $K_1 = K(\alpha)$.

\autosyncline{1255}Since $H^q(k_1, \mu_2) \neq 0$,
Corollary \ref{milnor-cor} implies that there exists
an anisotropic $2^q$-dimensional Pfister form $Q$ over $k_1$.
The coefficients of $Q$ are algebraic over $F$,
since $k_1/k$ and $k/F$ are algebraic extensions.

\autosyncline{1261}Let $F_1 \subseteq k_1$ be the field obtained by adjoining the coefficients of $Q$ to $F$.
Choose $\beta \in F_1$ such that $F_1 = F(\beta)$.
By Proposition~\ref{reshensel}, we can identify $k_1$
with a subfield of the henselisation $K_1\hens$.
So we have the following chain of field extensions: $F \subseteq F_1 \subseteq k_1 \subseteq K_1\hens$.
\autosyncline{1266}Therefore, we can see $\beta$ as an element of $K_1\hens$ and
define $K' := K_1(\beta)$.
Since $K'$ is a subfield of $K_1\hens$,
the residue field $k' := k_1$ and value group $\Gamma$
will remain the same if we take an extension of $v$ to $K'$.
\autosyncline{1271}Let $F' \supseteq F_1$ be a maximal subfield of $K'$ on which $v$ is trivial.

\autosyncline{1273}We claim that the conditions of Main Theorem~\ref{main-valued1} are satisfied for $K'$,
with maximal subfield $F'$ and residue field $k'$.
The residue field still has characteristic zero
and the value group stayed the same,
so conditions~(\ref{cond1-char}) and (\ref{cond1-gamma}) are still satisfied.

\autosyncline{1279}We have the quadratic form $Q$ which is anisotropic over $k' = k_1$.
We made sure that the coefficients of $Q$ lie in $F_1 \subseteq F'$, by adjoining them.

\autosyncline{1282}By construction, $k'$ is a finite extension of $k$,
so we have $\cdim_2(F) = \cdim_2(k') = q$.
Since $k'/F'$ and $F'/F$ are algebraic, we must also have $\cdim_2(F') = q$.

\autosyncline{1286}On the other hand,
from $\cdim_2(F') = q$ it follows that $\cdim_2(F'(Z)) = q+1$.
By definition of cohomological dimension,
we have $H^{q+2}(L, \mu_2) = 0$ for all finite extensions $L$ of $F'(Z)$,
which implies that all Pfister forms over $L$ of dimension $2^{q+2}$ will be isotropic.

\autosyncline{1292}Using Main Theorem~\ref{main-valued1}, this would prove undecidability for $K'(C)$.
However, $[K':K]$ is finite, therefore one can make a model of $K'(C)$ in $K(C)^{[K':K]}$.
So undecidability for the finite extension $K'(C)$ implies undecidability for $K(C)$.

\autosyncline{1296}Conversely, suppose that condition \eqref{cond1-qf} holds.
The second part of this condition says that $H^{q+2}(L, \mu_2) = 0$
for all finite extentions $L$ of $F(Z)$.
This implies $\cdim_2(F(Z)) \leq q + 1$,
and Proposition \ref{cohom-ext} gives $\cdim_2(F) = \cdim_2(F(Z)) - 1 \leq q$.

\autosyncline{1302}The existence of an anisotropic $2^q$-dimensional Pfister form over $k$
implies that $H^q(k, \mu_2) \neq 0$ and $\cdim_2(k) \geq q$.
But $k$ is algebraic over $F$, so by Proposition~\ref{cohom-ext} we have the inequalities
$$
	q \leq \cdim_2(k) \leq \cdim_2(F) \leq q
$$
\autosyncline{1308}which imply $\cdim_2(F) = \cdim_2(k) = q$, hence finite.
\end{proof}

\autosyncline{1311}Note that the inequality ``$\cdim_2(F) \geq \cdim_2(k)$'' is always satisfied,
because $k$ is an algebraic extension of $F$ (see Proposition~\ref{resnonhensel}).
So, it suffices to check that $\cdim_2(F) \leq \cdim_2(k)$.

\subsection{The curve $C$}\label{sec-curve}

\autosyncline{1319\\funcfield0}In Main Theorem~\ref{main-valued1}, we assumed that $C$ had a rational point.
But we can easily get rid of this condition using field extensions.

\begin{prop}\label{wlog-curve}
\autosyncline{1323}The conclusion of Main Theorem~\ref{main-valued1}
still holds if $C$ does not have a $K$-rational point.
\end{prop}

\begin{proof}
\autosyncline{1328}We use the formulation of condition \eqref{cond1-qf} as in Proposition~\ref{wlog-qf},
so we assume that $\cdim_2(F) = \cdim_2(k) < \infty$.

\autosyncline{1331}Over an algebraically closed field, $C$ must have a point
so let $P \in C(\bar{K})$.
Then $P$ is actually defined over a finite extension $K'$ of $K$.
Take an extension of $v$ to $K'$ and let $\Gamma'$ denote the new value group,
$k'$ the residue field and $F'$ a maximal subfield of $K'$ extending $F$.

\autosyncline{1337}We will now apply Main Theorem~\ref{main-valued1} for $K'$.
The value group $\Gamma'$ cannot be $2$-divisible
since $[\Gamma':\Gamma]$ is finite.
Since all extensions are finite,
$\cdim_2(F') = \cdim_2(F)$ and $\cdim_2(k') = \cdim_2(k)$,
\autosyncline{1342}therefore $\cdim_2(F') = \cdim_2(k') < \infty$, proving the new condition~\eqref{cond1-qf}.
Now $P$ is a $K'$-rational point,
so Main Theorem~\ref{main-valued1} gives undecidability for $K'(C)$, hence also for $K(C)$.
\end{proof}

\subsection{Second version of the Main Theorem}

\autosyncline{1351\\funcfield0}Applying the previous two sections,
we can reformulate Main Theorem~\ref{main-valued1} as follows:

\begin{mtheorem}\label{main-valued2}
\autosyncline{1355}Let $K$ be a field of characteristic zero with a valuation $v: K^* \onto \Gamma$.
Let $\calO$ denote the valuation ring and $k$ the residue field.

\autosyncline{1358}Assume the following conditions are satisfied:
\begin{enumerate}[(i)]
\item\label{cond-char} The characteristic of the residue field $k$ is zero.
\item\label{cond-gamma} The value group $\Gamma$ is not $2$-divisible.
\item\label{cond-qf}
	\autosyncline{1363}Let $F$ be a maximal field contained in $\calO$.
	The $2$-cohomological dimensions of $F$ and $k$ are equal and finite.
\end{enumerate}

\autosyncline{1367}Let $C$ be a smooth projective geometrically connected curve defined over $K$
and let $K(C)$ be its function field.
Then there exists a diophantine model of $\Z$ over $K(C)$
with coefficients in some finitely generated subfield $\calL_0$ of $K(C)$.
\end{mtheorem}

\section{Coefficient field}\label{sec-vallang}

\autosyncline{1377\\funcfield0}So far, we have not really discussed the field $\calL_0$ of coefficients
for which we have undecidability of diophantine equations.
We start from $\Q$ and add some constant symbols
to make our diophantine model of $\Z$.
There are four places in the proof where we need to enlarge $\calL_0$:

\begin{enumerate}
\item
\autosyncline{1385}To define the extension $L$ and the points $P_1$ and $P_2$ on $E(L)$,
$\calL_0$ must at least contain $T$ and $Z$.
For $T$ any element from $K$ having positive odd valuation will do,
$Z$ is simply a transcendental element over $K$ generating $K(Z)$.

\item
\autosyncline{1391}To apply Proposition~\ref{wlog-qf},
we might need to extend our field $K$ to a finite extension $K' = K(\alpha, \beta)$.
So we need the coefficients of the minimal polynomial of $\alpha$ and $\beta$ in $\calL_0$.
From the proof of Proposition~\ref{wlog-qf},
it can be seen that these are algebraic over $F$.
\autosyncline{1396}So, if $F$ happens to be finitely generated over $\Q$,
we might as well include all of $F$ into $\calL_0$.

\item
\autosyncline{1400}We have to express the coefficients of the quadratic form $Q$.
These will also be algebraic over $F$.

\item
\autosyncline{1404}Finally, we might need a finite extension to apply Proposition~\ref{wlog-curve}.
\end{enumerate}

\autosyncline{1407}In concrete examples, one can usually specify the field $\calL_0$ explicitly,
see some of the examples below.

\section{Examples}\label{sec-examples}

\autosyncline{1414\\funcfield0}In this section we give some examples for which our theorem can be applied.
We recover many known results.

\autosyncline{1417}The first example shows that we might as well take function fields
of arbitrary varieties (of dimension $\geq 1$) instead of curves.

\begin{example}\label{example-fg}
\autosyncline{1421}Let $K$ be such that the conditions of Main Theorem~\ref{main-valued2} are satisfied
for some curve $C$.
Let $L$ be a finitely generated extension of $K$,
with transcendence degree at least $1$.
Then HTP for $L$ has a negative answer
(\autosyncline{1426}for some finitely generated field $\calL_0$).
\end{example}

\begin{proof}
\autosyncline{1430}We consider two cases, according to the transcendence degree of $L/K$.

\autosyncline{1432}If the transcendence degree is exactly $1$
then we let $K'$ be the algebraic closure of $K$ inside $L$.
Then $L$ is the function field of a curve over $K'$,
let $L = K'(C')$.

\autosyncline{1437}Let $v$ be an extension of the given valuation to $K'$.
The new value group $\Gamma'$ might be larger than the original $\Gamma$,
but in any case $[\Gamma':\Gamma]$ is finite,
so $\Gamma'$ will still be non-$2$-divisible.

\autosyncline{1442}The maximal subfield $F' \supseteq F$ of $\calO' \subseteq K'$
will be a finite extension of $F$, so $\cdim_2(F') = \cdim_2(F)$.
The same is true for the new residue field $k'$,
so $\cdim_2(F') = \cdim_2(k') < \infty$.

\autosyncline{1447}If $L$ has transdendence degree $\geq 2$ over $K$,
then we take a transcendence basis $\{Z_1, \dots, Z_n\}$ of $L/K$.
Let $u$ be a valuation on $K(Z_1, \dots, Z_{n-1})$ with residue field $K$.
Let $v$ be the given valuation on $K$.
Let $w$ be the composition of $u$ with $v$
(\autosyncline{1452}see Proposition~\ref{compval} but with $u$ and $v$ swapped).
We want to show that the conditions of Main Theorem~\ref{main-valued2}
are satisfied for the base field $K(Z_1, \dots, Z_{n-1})$ with valuation $w$
and the curve $C = \PS{1}$.
Then the statement for $L$ will follow from the first part of this proof.

\autosyncline{1458}It is easy to see that $F \subseteq \calO_v \subseteq K$ is also a maximal subfield of $\calO_w$.
Proposition~\ref{compval} says that the residue field of $w$ is $k$.
So, clearly conditions (\ref{cond-char}) and (\ref{cond-qf}) are satisfied.
Also condition (\ref{cond-gamma}) is satisfied because of the exact sequence \eqref{compval-seq}
and the fact that $\Gamma_u$ is not $2$-divisible.
\end{proof}

\autosyncline{1465}To simplify the following examples,
we will only consider rational function fields.
However, because of the preceding example, everything still works
for function fields of varieties.
Moreover, considering only rational function fields makes
\autosyncline{1470}the examples more concrete such that one can specify $\calL_0$ in certain cases.

\begin{example}\label{exFZ1Z2}
\autosyncline{1473}If $F$ is a characteristic zero field with $\cdim_2(F)$ finite,
then HTP for the $2$-variable rational function field $F(Z_1, Z_2)$ has a negative answer.
\end{example}

\begin{proof}
\autosyncline{1478}Apply the theorem with $K = F(Z_1)$ and $v$ the discrete valuation associated to $Z_1$,
which has residue field $F$.
\end{proof}

\autosyncline{1482}Applying Example~\ref{example-fg}, this last example can be generalized
to function fields of varieties of dimension at least $2$ over $F$.

\begin{example}
\autosyncline{1486}If $F$ is a number field,
then HTP for $F(Z_1, Z_2)$ has a negative answer
with $\calL_0 = \Q(Z_1, Z_2)$.
(see also \cite{kim-roush-padic}).
\end{example}

\begin{proof}
\autosyncline{1493}From the Theorem of Hasse--Minkowski it follows that
all $4$-dimensional quadratic forms over a non-real
number field are isotropic.
On the other hand, over a real field there are anisotropic Pfister forms
of arbitrarily high dimension: take $\form{1,1} \otimes \form{1,1} \otimes \ldots$.
\autosyncline{1498}Using the results mentioned in Section \ref{seccohom},
this implies that $\cdim_2(F) = \infty$ if $F$ is a real number field
and $\cdim_2(F) = 2$ otherwise.
So in the non-real case we just have to apply Example~\ref{exFZ1Z2}.

\autosyncline{1503}If $F$ is real, we can take the finite extension $F' = F(\sqrt{-1})$.
Then Main Theorem~\ref{main-valued2} gives undecidability for $F'(Z_1, Z_2)$,
which implies undecidability for $F(Z_1, Z_2)$.
\end{proof}

\begin{example}
\autosyncline{1509}HTP for $\R(Z_1, Z_2)$ and $\C(Z_1, Z_2)$ has a negative answer
with $\calL_0 = \Q(Z_1, Z_2)$.
(for $\R$ see also \cite{denef-real},
for $\C$ see also \cite{kim-roush-ct1t2}).
\end{example}

\begin{example}
\autosyncline{1516}Let $F$ be a characteristic zero field with $\cdim_2(F)$ finite.
Then HTP for $F((T))(Z)$ has a negative answer.
\end{example}

\begin{proof}
\autosyncline{1521}Let $K = F((T))$ and let $v$ be the discrete valuation at $T$.
The valuation ring $\calO = F[[T]]$ has $F$ as maximal subfield.
This way, the conditions for Main Theorem~\ref{main-valued2} are satisfied.
\end{proof}

\autosyncline{1526}This example can be generalized somewhat:

\begin{example}\label{ex-completion}
\autosyncline{1529}Let $K$ be a field for which the conditions of Main Theorem~\ref{main-valued2} are satisfied.
Let $K'$ be any extension of $K$, contained in the maximal completion $\hat{K}$
(for discrete valuations, this is ``the'' completion).
Then HTP for $K'(Z)$ has a negative answer.
\end{example}

\begin{proof}
\autosyncline{1536}Extend the given valuation $v$ to a valuation on $K'$.
The residue field and value group will remain the same
($\hat{K}$ is the maximal field with this property).
In general, the maximal subfield $F'$ of $\calO'$ could be an extension of $F$,
but still contained in $k$.
\autosyncline{1541}Since $F \subseteq F' \subseteq k$ and $k/F$ is algebraic,
the extensions $k/F'$ and $F'/F$ are also algebraic.
Hence
$$
	q = \cdim_2(k) \leq \cdim_2(F') \leq \cdim_2(F) = q
$$
\autosyncline{1547}from which $\cdim_2(F') = \cdim_2(k) = q$.
\end{proof}

\begin{example}
\autosyncline{1551}If $K$ is henselian, then we have $\cdim_2(F) = \cdim_2(k)$ by Proposition~\ref{reshensel}.
We still need to check the finiteness of $\cdim_2(k)$ though.
\end{example}

\begin{example}\label{ex-trans}
\autosyncline{1556}Let $F$ be a characteristic zero field for which $\cdim_2(F)$ is finite.
Let $\{X_i\}_{i \in I}$ be a set of algebraically independent variables,
with $\# I \geq 2$.
Then HTP for $F(\{X_i\}_{i \in I})$ has a negative answer.
\end{example}

\begin{proof}
\autosyncline{1563}Choose a well-ordering $\preccurlyeq$ on $I$, this is a total order on $I$ such that
every non-empty subset of $I$ has a minimal element
(the existence of well-orderings is equivalent to the axiom of choice).
$I$ itself also has a smallest element $i_0$, let $Z := X_{i_0}$.
We also define $I' := I \setminus \{i_0\}$ and $K := F(\{X_i\}_{i \in I'})$.
\autosyncline{1568}We have to prove undecidability for $F(\{X_i\}_{i \in I}) = K(Z)$.

\autosyncline{1570}Let
$$
	\Gamma := \bigoplus_{i \in I'} \Z. \qquad \text{(direct sum of abelian groups)}
$$
Since $\# I \geq 2$, this $\Gamma$ is not $2$-divisible.

\autosyncline{1576}We make this into an ordered abelian group $\Gamma, +, \leq$
by using the lexicographic ordering coming from $I, \preccurlyeq$.
In detail:
let $\gamma = \oplus_{i \in I'} \gamma_i \in \Gamma$.
Assume $\gamma \neq 0$ and look at the set
$J \subseteq I'$ \autosyncline{1581}of all $i$ such that $\gamma_i \neq 0$.
Let $j_0$ be the minimal element from $J$,
and define $0 < \gamma$ if and only if $0 < \gamma_{j_0}$.

\autosyncline{1585}To define a valuation $v: K^* \onto \Gamma$,
we let $v$ be trivial on $F$ and define $v$ for monomials:
$$
	v\left(\prod_{i \in I'} X_i^{m_i}\right) = \bigoplus_{i \in I'} m_i \in \Gamma
.$$
\autosyncline{1590}Then the valuation of a polynomial is defined to be the minimal valuation of its terms.
Finally, for rational functions we define $v(x/y) = v(x) - v(y)$.
One can check that this does indeed satisfy the axioms of a valuation,
and that the residue field is $F$ (hence $\cdim_2(k) = \cdim_2(F) < \infty$).
\end{proof}

\begin{example}
\autosyncline{1597}Let $K$ be a field of characteristic zero containing an algebraically closed subfield.
If $K$ admits a valuation with non-$2$-divisible value group
and residue characteristic zero,
then HTP for $K(Z)$ has a negative answer with $\calR_0 = \Q(T,Z)$,
where $T$ can be any element with odd valuation.
\end{example}

\begin{proof}
\autosyncline{1605}Remark that $K$ cannot be algebraically closed itself,
because all valuations on algebraically closed fields
have divisible value groups.

\autosyncline{1609}Write $v$ for the given valuation with value group $\Gamma_v$,
valuation ring $\calO_v$,
maximal subfield $F_v \subseteq \calO_v$
and residue field $k_v$.
Let $C$ be an algebraically closed subfield of $F_v$
(\autosyncline{1614\\funcfield0}one can always take $C = \bar{\Q}$,
since $\bar{\Q}$ has no non-trivial valuations with residue characteristic zero).

$C$ \autosyncline{1617}is contained in $F_v$, so it is also contained in $k_v$.
We would like to define a valuation $u$ on $k_v$ with $C$ as residue field,
we do this as follows:
Choose a transcendence basis $\{X_i\}_{i \in I}$ for $k_v$ over $C$.
As in Example~\ref{ex-trans}, we can construct a valuation $u$ on $C(\{X_i\}_{i \in I})$
\autosyncline{1622}with residue field $C$.
Extend this valuation to $k_v$.
This extension is algebraic, so the new residue field is an algebraic extension of $C$,
hence $C$ itself.

\autosyncline{1627}Let $w$ be the composite valuation of $v$ and $u$,
as defined in Proposition~\ref{compval}.
We would like the apply the Main Theorem on $K$ with valuation $w$.
Since $\Gb_v$ is not $2$-divisible,
the exact sequence \eqref{compval-seq} ensures that $\Gb_w$ is not $2$-divisible either.

\autosyncline{1633}We claim that $C$ is a subfield of $\calO_w$.
We know that $C^* \subseteq \calO_u^*$, and since $\pi_v$ is an isomorphism on $C$,
we also have $C^* \subseteq \pi_v\inv(\calO_u^*) = \calO_w^*$.

\autosyncline{1637}The residue field of $w$ is $C$, so $C$ must be a maximal subfield of $\calO_w$.
We have $\cdim_2(C) = \cdim_2(C) = 0$,
so we can apply Main Theorem~\ref{main-valued2} with the valuation $w$.
\end{proof}

\bibliographystyle{amsalpha}
\bibliography{all}

\end{document}